\documentclass[11pt]{amsart}
\usepackage{geometry}                
\usepackage{graphicx}
\usepackage{amssymb}
\usepackage{epstopdf}
\usepackage{verbatim} 
\usepackage{color}
\def\R{\mathbb{R}}
\def\eps{\varepsilon}
\def\reals{\mathbb{R}}

\def\C{\mathbb{C}}
\def\N{\mathbb{N}}

\def\B{\mathcal{B}}

\def\G{\mathcal{G}}
\def\P{\mathcal{P}}

\newtheorem{theorem}{Theorem}
\newtheorem{corollary}[theorem]{Corollary}
\newtheorem{remark}[theorem]{Remark}
\newtheorem{lemma}[theorem]{Lemma}

\author{Michael Christ}
\address{
        Michael Christ\\
        Department of Mathematics\\
        University of California \\
        Berkeley, CA 94720-3840, USA}
\email{mchrist@math.berkeley.edu}
\author{Diogo Oliveira e Silva}
\address{
        Diogo Oliveira e Silva\\
        Department of Mathematics\\
        University of California \\
        Berkeley, CA 94720-3840, USA}
\email{dosilva@math.berkeley.edu}
\thanks{The first author was supported by NSF grant DMS-0901569. 
The second author was supported by the Funda\c{c}\~{a}o para a Ci\^{e}ncia e a Tecnologia 
(FCT/Portugal grant SFRH/BD/28041/2006).
Any opinions, findings, and conclusions
or recommendations expressed in this paper are those of the authors
and do not necessarily reflect the views of the National Science Foundation.}

\title{On Trilinear Oscillatory Integrals}
\date{May 5, 2010}                 

\begin{document}
\maketitle

\begin{abstract}
We examine a certain class of trilinear integral operators which incorporate oscillatory factors $e^{iP}$, where $P$ is a real-valued polynomial, 
and prove smallness of such integrals in the presence of rapid oscillations.
\end{abstract}

\section{introduction}

This note continues the study of
multilinear oscillatory integral expressions of the form
$$I(\lambda P;f_1,\ldots, f_n)=\int_{\R^m}e^{i\lambda P(x)}\prod_{j=1}^n f_j\circ\pi_j(x)\eta(x)dx,$$
where $\lambda\in\R$ is a parameter, $P:\R^m\rightarrow\R$ is a real-valued polynomial, $\pi_j:\R^m\rightarrow V_j$ are orthogonal projections onto some subspaces $V_j$ of $\R^m$, $f_j:V_j\rightarrow\C$ are locally integrable functions with respect to Lebesgue measure on $V_j$, and $\eta\in C_0^1(\R^m)$ is compactly supported. All the subspaces $V_j$ are assumed to have the same dimension, which is denoted by $\kappa$. 

Christ, Li, Tao, and Thiele \cite{CLTT} 
initiated this study, exploring
conditions on the polynomial phase $P$ and on the projections $\{\pi_j\}$ which ensure decay estimates of the form
\begin{equation} \label{decay}
|I(\lambda P; f_1,\ldots,f_n)|\leq C\langle\lambda\rangle^{-\epsilon}\prod_{j=1}^n \|f_j\|_{L^\infty(V_j)}.
\end{equation}
Their results were restricted to the comparatively extreme cases $\kappa=1$ and $\kappa=m-1$, 
and the small codimension case $n\leq\frac{m}{m-\kappa}$, leaving most cases open.

In the present paper we consider the trilinear situation in $\R^m=\R^{2 \kappa}$ for arbitrary $\kappa\ge 2$. 
A typical expression of this type is then
$$I(P;f_1,f_2,f_3)=\iint_{\R^{2 \kappa}}e^{i P(x,y)}f_1(x)f_2(y)f_3(x+y)\eta(x,y)dxdy,$$
with coordinates $(x,y)\in\R^{\kappa + \kappa }$.
Before stating our main theorem, we introduce some notation and recall relevant results from the literature.

\subsection{Review}
Let $d\geq 1$ be a positive integer, and let $\{\pi_j\}_{j=1}^3$ be surjective linear
mappings from $\R^{2 \kappa}$ to $\reals^ \kappa $. 
A polynomial $P:\R^{2 \kappa}\to\R$ is said to be degenerate (with respect to the projections $\{\pi_j\}_j$) if there exist polynomials $p_j:\R^\kappa\rightarrow\R$ such that $P=\sum_{j=1}^3 p_j\circ\pi_j$. 
The vector space of all degenerate polynomials $P:\R^{2 \kappa}\rightarrow\R$ of degree $\leq d$ is a subspace $\P_{degen}$ of the vector space $\P(d)$ of all polynomials $P:\R^{2 \kappa}\rightarrow\R$ of degree $\leq d$. Denote the quotient space by $\P(d)/\P_{degen}$, by $[P]$ the equivalence class of $P$ in $\P(d)/\P_{degen}$, and by $\|\cdot\|_{nd}$ some fixed choice of norm for this quotient space. In a similar way, let $\|\cdot\|_{nc}$ denote some fixed choice of norm for the quotient space of polynomials $P:\R^{2 \kappa}\rightarrow\R$ of degree $\leq d$ modulo constants.

It will be convenient to work with norms defined by inner products.
If $P(x,y)=\sum_{\alpha,\beta} c_{\alpha\beta}x^\alpha y^\beta$, then we set
$$\|P\|_{\P(d)}=\Big(\sum_{\alpha,\beta} |c_{\alpha\beta}|^2\Big)^{1/2},\;\;\|P\|_{nc}=\Big(\sum_{(\alpha,\beta)\neq (0,0)} |c_{\alpha\beta}|^2\Big)^{1/2}.$$
$\|\cdot\|_{nd}$ is defined by choosing some Hilbert space structure for $\P(d)/\P_{degen}$.

The norm $\|\cdot\|_{nc}$ controls oscillatory integrals of the first kind, in light of the following version of stationary phase:

\begin{theorem}\label{sp}
Let $p(t)=\sum_{|\alpha|\leq d}c_\alpha t^\alpha, c_\alpha\in\R$, be a polynomial in $m$ variables of degree $d\geq 1$. Then 
$$\Big|\int_{[0,1]^m} e^{ip(t)}dt\Big|\leq C_{d,m}\Big(\sum_{0<|\alpha|\leq d} |c_\alpha|\Big)^{-1/d}.$$
\end{theorem}

Theorem \ref{sp} is a straightforward consequence of  the well-known lemma of van der Corput \cite{S}.

On the other hand, the norm $\|\cdot\|_{nd}$ controls multilinear oscillatory integrals (in particular, oscillatory integrals of the second kind), as is shown in \cite{CLTT}. The following theorem from \cite{CLTT} 
is most relevant to our discussion:

\begin{theorem}\label{l2}
Suppose that $n<2m$ and $d<\infty$. Then, for any family $\{V_j\}_{j=1}^n$ of one-dimensional subspaces of $\R^m$ which lie in general position, there exist constants $C<\infty$ and $\epsilon>0$ such that
$$|I(P;f_1,\ldots,f_n)|\leq C\langle \|P\|_{\text{nd}}\rangle^{-\epsilon}\prod_{j=1}^n\|f_j\|_{L^2(V_j)}$$
 for all polynomials $P:\R^m\rightarrow\R$ of degree $\leq d$  and for all functions $f_j\in L^2(\R)$.
Moreover, $\epsilon$ can be taken to depend only on $n,m$ and $d$. 
\end{theorem}

\subsection{Result}
Let $\{\pi_j: 1\le j\le 3\}$ be a collection of three surjective linear mappings from $\reals^{2 \kappa}$ to $\reals^ \kappa$. 
We say that these lie in general position if for any two indices $i\ne j\in\{1,2,3\}$,
the nullspace of $\pi_i$ is transverse to the nullspace of $\pi_j$.

In the present paper we prove the following:
\begin{theorem}\label{thm}
Let $\kappa\ge 1$ and $d<\infty$.
Let $\{\pi_j: 1\le j\le 3\}$ be a collection of three surjective linear mappings from $\reals^{2 \kappa}$ to $\reals^ \kappa $, 
which lie in general position.  
Then
$$|I(P;f_1,f_2,f_3)|\leq C\langle\|P\|_{nd}\rangle^{-\epsilon}\prod_{j=1}^3 \|f_j\|_{L^2(\R^ \kappa)},$$
for all polynomials $P:\R^{2 \kappa}\rightarrow\R$ of degree $\leq d$ and
for all functions $f_j\in L^2(\R^ \kappa)$, with constants $C,\epsilon\in\R^+$ which depend only on $\kappa, d$ and $\eta$.
\end{theorem}

A more general result is established in \cite{snarl}, but we hope that the quite different method of the present paper
will be of some value.

If $x,y$ are real numbers, we will write $x\lesssim y$ if there exists a finite constant $C$ such that $x\leq Cy$. The constant $C$ may depend on some parameters which will be clear from the context.

\section{First reduction}

It is no loss of generality to restrict attention to the case where 
$\R^{2 \kappa}$ is identified with $\R^ \kappa_x\times\R^ \kappa_y$, and
$\pi_1(x,y)=x$, $\pi_2(x,y)=y$, and $\pi_3(x,y)=x+y$.
Indeed,
since the nullspaces of $\pi_1,\pi_2$ are transverse, we may adopt coordinates
$(x,y)\in\reals^{\kappa + \kappa}$ such that the nullspace of $\pi_1$ is $\{(0,y)\}$,
while the nullspace of $\pi_2$ is $\{(x,0)\}$.
Writing $\pi_3(x,y)=Ax+By$ where $A,B:\R^ \kappa\to\R^ \kappa $ are linear,
the transversality hypothesis implies that both $A,B$ are injective.
Therefore it is possible to make invertible changes of coordinates in $\R^\kappa_x,\R^ \kappa_y$
so that $A,B$ become the identity operator;
$\pi_3(x,y)=x+y$. 
Next, $\pi_1(x,y)=Dx$ for some invertible $D:\R^ \kappa\to\R^ \kappa $.
By making a change of variables in the range of $D$,
we may achieve $\pi_1(x,y)\equiv x$.
Finally, a corresponding change of coordinates in the range of $\pi_2$
makes $\pi_2(x,y)\equiv y$.

In order to keep the notation simple, we will discuss in detail the case $\kappa =2$,
then will indicate in \S~\ref{section:higher} how the analysis extends without additional
difficulty to arbitrary dimensions.

\section{Second reduction}
We will restrict our attention to polynomial phases of the form $\lambda P$, where $\lambda\in (0,\infty)$
and $\|P\|_{\text{nd}}=1$; for if $\|P\|_{nd}=0$, then the conclusion of Theorem \ref{thm} is trivial. In particular, $P$ will henceforth be assumed nondegenerate with respect to the projections $\{\pi_j\}_{j=1}^3$. 

In the following lemma, 
$P_{(x_2,y_2)}(x_1,y_1):=P(x_1,y_1,x_2,y_2)$,
and $\|\cdot\|_{nd}$ denotes a norm on the space of polynomials of degree $\leq d$ in $x_1$ and $y_1$ modulo degenerate polynomials with respect to the projections $(x_1,y_1)\mapsto x_1,y_1,x_1+y_1$.
We will sometimes write $I(P)$ as shorthand for $I(P;f_1,f_2,f_3)$.

Let $K\subset\reals^2_{x_2,y_2}$ be the projection of the support of $\eta$ onto the $(x_2,y_2)$ plane.
\begin{lemma}\label{first}
Let ${P}:\R^4\rightarrow\R$ be a real-valued polynomial of degree $\leq d$.
If the polynomial $(x_1,y_1)\mapsto P(x,y)$ is nondegenerate with respect to the one-dimensional projections $(x_1,y_1)\mapsto x_1,y_1,x_1+y_1$ for some $(x_2,y_2)\in\R^2$, then there exists a constant $C<\infty$ such that:
$$|I(\lambda P; f_1,f_2,f_3)|\leq C(\lambda\sup_{(x_2,y_2)\in K}\|
P_{(x_2,y_2)}\|_{nd})^{-\sigma}\prod_{j=1}^3\|f_j\|_2,$$ 
for all functions $f_j\in L^2(\R^2)$, where $\sigma$ is a constant which depends only on $d$. 
\end{lemma}

\begin{proof}
We can apply Theorem \ref{l2} with $m=2$ and $n=3$ to conclude that
$$J_{x_2,y_2}(\lambda P)
:=\iint_{\R^2}e^{i\lambda P(x_1,y_1,x_2,y_2)}
f_1(x_1,x_2)f_2(y_1,y_2)f_3(x_1+y_1,x_2+y_2)\eta(x_1,y_1,x_2,y_2)dx_1dy_1$$
satisfies
\begin{align*}
|J_{x_2,y_2}(\lambda P)|\leq& C(1+\lambda^2|Q(x_2,y_2)|)^{-\rho}\|f_1(\cdot,x_2)\|_2\|f_2(\cdot,y_2)\|_2\|f_3(\cdot,x_2+y_2)\|_2\\
=&C(1+\lambda^2|Q(x_2,y_2)|)^{-\rho}g_1(x_2)g_2(y_2)g_3(x_2+y_2),
\end{align*}
for some $\rho>0$ depending only on $d$,
where $g_j(t)=\|f_j(\cdot,t)\|_2$
and $Q(x_2,y_2)=\|P_{(x_2,y_2)}(\cdot)\|_{\text{nd}}^2$
is a polynomial of degree $\leq 2d$. 

For $\epsilon>0$, let 
$$E_\epsilon:=\{(x,y)\in K: |Q(x,y)|<\epsilon\}.$$
A basic sublevel set estimate \cite{BG} yields
$$|E_\epsilon|\leq C\|Q\|_{L^{\infty}(K)}^{-\delta'}\epsilon^{\delta'}\;\;\textrm{ for }\delta'=\frac{1}{\textrm{deg} (Q)}$$
and some absolute constant $C<\infty$ if $Q$ has positive degree.
We now split the original integral $I(\lambda P;f_1,f_2,f_3)$ 
into two pieces and estimate each of them separately. On the one hand, H\"{o}lder's and Young's inequalities imply:
\begin{align*}
\iint_{E_\epsilon}|J_{x_2,y_2}(\lambda P)|dx_2dy_2\leq&C\iint_{E_\epsilon}g_1(x_2)g_2(y_2)g_3(x_2+y_2)dx_2dy_2\\
\leq&C|E_\epsilon|^{1/4}\Big(\iint_{\R^2}g_1^{4/3}(x_2)g_2^{4/3}(y_2)g_3^{4/3}(x_2+y_2)dx_2dy_2\Big)^{3/4}\\
\leq& C|E_\epsilon|^{1/4}\prod_{j=1}^3\|g_j^{4/3}\|_{3/2}^{3/4}\\
=&C|E_\epsilon|^{1/4}\prod_{j=1}^3\|g_j\|_2
\\
\leq & C\|Q\|_{L^{\infty}(K)}^{-\delta}\epsilon^{\delta}\prod_{j=1}^3\|f_j\|_2.
\end{align*}
On the other hand,
\begin{align*}
\iint_{\R^2\setminus E_\epsilon}|J_{x_2,y_2}(\lambda P)|dx_2dy_2\leq& C(1+\lambda^2\epsilon)^{-\rho}\iint_{\R^2}g_1(x_2)g_2(y_2)g_3(x_2+y_2)dx_2dy_2\\
\leq& C(1+\lambda^2\epsilon)^{-\rho}\prod_{j=1}^3\|g_j\|_{3/2}\\
\leq&C(1+\lambda^2\epsilon)^{-\rho}\prod_{j=1}^3\|f_j\|_{2}.\\
\end{align*}
If $Q$ has degree zero then the same conclusion is reached more simply, with $\epsilon^{\delta'}$
replaced by $1$.
Thus
$$|I(\lambda P; f_1,f_2,f_3)|\leq C\Big[\|Q\|_{L^{\infty}(K)}^{-\delta}\epsilon^{\delta}+(\lambda^2\epsilon)^{-\rho}\Big]\prod_{j=1}^3\|f_j\|_2.$$

\noindent Since $\|Q\|_{L^\infty(K)}=\sup_{(x_2,y_2)\in K}\|P_{(x_2,y_2)}\|^2_{nd}$, optimizing in $\epsilon$ yields an upper bound 

$$|I(\lambda P; f_1,f_2,f_3)|\leq C\Big(\lambda\sup_{(x_2,y_2)}\|P_{(x_2,y_2)}\|_{nd} \Big)^{-\frac{2\rho\delta}{\rho+\delta}}\prod_{j=1}^3\|f_j\|_2.$$
\end{proof}


\section{Third reduction}
The goal of this step is to show that it is enough to consider functions of the form $f_j(u_1,u_2)=e^{i\phi_j(u_1,u_2)}$, where $\phi_j$ has polynomial dependence on $u_1$ of bounded degree and is real-valued.

From the last section, we get the desired decay rate unless $P$ is ``almost degenerate'' with respect to the projections $(x_1,y_1)\mapsto x_1,y_1,x_1+y_1$ for almost every $(x_2,y_2)\in\R^2$, in the sense that
\begin{equation*}
\sup_{(x_2,y_2)\in K}\|P_{(x_2,y_2)}\|_{nd}\lesssim\lambda^{-1+\tau}\textrm{ for some }\tau>0.
\end{equation*}
We have the freedom to choose $\tau$ arbitrarily small later on in the argument. 

In this case, one can decompose 
$$P(x,y)=q_1(x_1,x_2,y_2)+q_2(y_1,x_2,y_2)+q_3(x_1+y_1,x_2,y_2)+R(x,y),$$
for some measurable functions $q_j$ and $R$ which are polynomials of degree $\leq d$ in $x_1$ and $y_1$, and where the remainder $R$ satisfies
$$|R(x,y)|\lesssim\lambda^{-1+\tau}\textrm{ if } (x,y)\in K'$$
for any fixed compact set $K'\subset\reals^{4}$.
To justify this, for each integer $k\ge 0$ choose
some Hilbert space norm  for the vector space of all homogeneous polynomials in $x_1$ and $y_1$ of degree $k$. 
Write $P_{(x_2,y_2)}(x_1,y_1)=P(x,y)$.
Express
$P(x,y)=\sum_{k=0}^d P_{k,(x_2,y_2)} (x_1,y_1)$ where $P_{k,(x_2,y_2)}$
is a homogeneous polynomial of degree $k$ in $(x_1,y_1)$,
whose coefficients are polynomials in $(x_2,y_2)$.
Now we use two facts  implicitly shown in \cite{CLTT}.
Firstly, if $p=\sum_k p_k$ is a decomposition of a polynomial in $(x_1,y_1)$ into its homogeneous summands of degree $k$,
then
$\sum_k \|p_{k}\|_{\text{nd}}$ is comparable to  $\|p\|_{\text{nd}}$.
Secondly,
if $p(x_1,y_1)$ is homogeneous of degree $k$, then for any $d\ge k$,
the norm of $p$ in the space of all homogeneous polynomials of degree $k$
modulo polynomials $cx_1^k+c'y_1^k+c''(x_1+y_1)^k$
is comparable to the norm $\|p\|_{\text{nd}}$
of $p$ in the space of all polyomials of degrees $\le d$ modulo all
degenerate polynomials of degrees $\le d$,
where degenerate polynomials are those which are sums of polynomials
in $x_1$, polynomials in $y_1$, and polynomials in $x_1+y_1$.
Qualitative versions of these two facts were established in \cite{CLTT};
the quantitative versions stated here follow from the equivalence of
all norms in any finite-dimensional vector space.

For each $k$, the projection $Q_{k,(x_2,y_2)}$ of $P_{k,(x_2,y_2)}$
onto the span of $x_1^k,y_1^k,(x_1+y_1)^k$ 
has polynomial dependence on $(x_2,y_2)$.
Moreover, all coefficients of $P_{k,(x_2,y_2)}-Q_{k,(x_2,y_2)}$ are $O(\lambda^{-1+\tau})$ for $(x_2,y_2)\in K$,
and therefore for $(x_2,y_2)$ in any fixed bounded set.
For $k\ge 2$, $Q_{k,(x_2,y_2)}$ decomposes uniquely as $q_{1,k}(x_2,y_2)x_1^k+q_{2,k}(x_2,y_2)y_1^k+q_{3,k}(x_2,y_2)(x_1+y_1)^k$;
these coefficients $q_{i,k}$ continue to have polynomial dependence on $(x_2,y_2)$.
For $k=1$ there is likewise a unique such decomposition, with the additional condition $q_{3,k}\equiv 0$,
and for $k=0$, with two additional conditions $q_{2,k}\equiv q_{3,k}\equiv 0$.
Recombining terms gives the claim.






Let us use this to work with $J_{x_2,y_2}$ (a similar calculation occurs in p.\ 15 of \cite{CLTT}):
\begin{align*}
J_{x_2,y_2}(\lambda P)=&\iint_{\R^2}e^{i\lambda P(x,y)}f_1(x)f_2(y)f_3(x+y)\eta(x,y)dx_1dy_1\\
=&\iint_{\R^2}e^{i\lambda q_1(x_1,x_2,y_2)}f_1(x_1,x_2)e^{i\lambda q_2(y_1,x_2,y_2)}f_2(y_1,y_2)
\\
& \qquad\qquad\qquad\qquad \cdot 
e^{i\lambda q_3(x_1+y_1,x_2,y_2)}f_3(x_1+y_1,x_2+y_2)e^{i\lambda R}\eta dx_1dy_1\\
=&\iint_{\R^2}g_1(x_1)g_2(y_1)g_3(x_1+y_1)\zeta(x_1,y_1) dx_1dy_1\\
=& C\iint\Big(\int\widehat{g}_1(\xi_1)e^{ix_1\xi_1}d\xi_1\Big)g_2(y_1)g_3(x_1+y_1)
\\
& \qquad\qquad\qquad\qquad \cdot
\Big(\iint\widehat{\zeta}(\xi_2,\xi_3)e^{i(x_1,y_1)\cdot(\xi_2,\xi_3)}d\xi_2d\xi_3\Big) dx_1dy_1\\
=& C\iiint\widehat{g}_1(\xi_1)\widehat{\zeta}(\xi_2,\xi_3)\Big[\int g_2(y_1)e^{iy_1\xi_3}\Big(\int g_3(x_1+y_1)e^{ix_1(\xi_1+\xi_2)}dx_1\Big)dy_1\Big] d\xi_1d\xi_2d\xi_3\\
=& C\iiint \widehat{g}_1(\xi_1)\widehat{g}_2(\xi_1+\xi_2-\xi_3)\widehat{g}_3(-\xi_1-\xi_2)\widehat{\zeta}(\xi_2,\xi_3)d\xi_1d\xi_2d\xi_3.\\
\end{align*}

Implicit in this notation is the dependence of the functions $g_j:=e^{i\lambda q_j}f_j$ and $\zeta:=e^{i\lambda R}\eta$ on $x_2$ and $y_2$.

Since $\lambda R$ is a polynomial in $x_1$ and $y_1$ of bounded degree which is $O(\lambda^{\tau})$ on supp$(\eta)$ and the same holds for all its derivatives, we have that
$$|\widehat{\zeta}(\xi)|\leq C_{n,\eta}\lambda^{n\tau}(1+|\xi|)^{-n},\;\forall\xi\in\R^2,\;\forall n\in\N,$$

\noindent provided $\eta\in C^n$. 
In particular, if\footnote{We lose no generality in assuming this extra smoothness on $\eta$: by the usual decomposition of a compactly supported H\"{o}lder continuous function $\zeta=f+g$ into a smooth part $f$ such that $\|f\|_{C^n}=O(\lambda^{C_n\delta})$ and a bounded remainder $g$ such that $\|g\|_\infty=O(\lambda^{-\delta})$, it is easy to see that if the result holds for some $\eta\in C_0^n$ ($n\in\N$) with a constant $C=O(\|\eta\|_{C^n})$, then it will continue to hold 
for all $\eta$ which are compactly supported and H\"older continuous of order $\alpha$.}
$\eta\in C_0^3(\R^4)$, then 
$$|\widehat{\zeta}(\xi)|\leq C\lambda^{3\tau}(1+|\xi|)^{-3}, \forall \xi\in\R^2.$$ 
We use this (together with Cauchy-Schwarz and Plancherel) to conclude that

\begin{align*}
|J_{x_2,y_2}(\lambda P)|\leq& C\lambda^{3\tau}\|\widehat{g}_1\|_{\infty}\iiint\frac{|\widehat{g}_2(\xi_1+\xi_2-\xi_3)||\widehat{g}_3(-\xi_1-\xi_2)|}{(1+|(\xi_2,\xi_3)|)^3}d\xi_1d\xi_2d\xi_3\\
\leq& C\lambda^{3\tau} \|\widehat{g}_1\|_{\infty}\|g_2\|_2\|g_3\|_2.
\end{align*}

Let $\delta>3\tau$ and consider the set:
$$F:=\{(x_2,y_2)\in\R^2: \|\widehat{g}_1\|_{\infty}\lesssim\lambda^{-\delta}\}.$$
There are two possibilities: 
\begin{itemize}
\item[(i)] If $|F^{\complement}|\lesssim\lambda^{-\delta}$, then $|I(\lambda P)|\lesssim\lambda^{-(\delta-3\tau)}$, as is easily seen by splitting the integral 
$$I(\lambda P)=\iint_{\R^2}J_{x_2,y_2}(\lambda P)dx_2dy_2$$ 
into the regions $F$ and $F^\complement$. 
\item[(ii)] If $|F^\complement|\gtrsim\lambda^{-\delta}$, we set $E:=F^\complement$. Note that $\|\widehat{g}_1\|_{\infty}\gtrsim\lambda^{-\delta}$ for every $(x_2,y_2)\in E$. 
\end{itemize}
Since condition (i) yields the desired decay, we 
restrict attention henceforth to the case in which
condition (ii) holds. Then there exists a measurable subset $E\subset\R^2$ such that $|E|\gtrsim \lambda^{-\delta} $ and $\|\widehat{g}_1\|_{\infty}\gtrsim\lambda^{-\delta}$ for every $(x_2,y_2)\in E$. 
We still have the freedom to choose $\delta>0$ as small as we wish later on in the argument.

Why is this conclusion of interest? Since

$$\lambda^{-\delta}\lesssim\|\widehat{g}_1\|_{\infty}=\sup_\xi\Big|\int_\R e^{i\lambda q_1(x_1,x_2,y_2)}f_1(x_1,x_2)e^{-ix_1\xi}dx_1\Big|,$$
we can find measurable functions $\theta$ and $\widetilde{\theta}$ such that, for $(x_2,y_2)\in E$, 

\begin{align*}
\lambda^{-\delta}\lesssim&\Big|\int e^{i\lambda q_1(x_1,x_2,y_2)}f_1(x_1,x_2)e^{-ix_1\theta(x_2,y_2)}dx_1\Big|\\
=&e^{-i\widetilde{\theta}(x_2,y_2)}\int f_1(x_1,x_2)e^{i\lambda q_1(x_1,x_2,y_2)-ix_1\theta(x_2,y_2)}dx_1.\\
\end{align*}

Because we are working in a fixed bounded region, there exist a measurable subset $E_1\subset\R$ such that $|E_1|\gtrsim\lambda^{-\delta}$ and a single number $\overline{y_2}\in\R$ such that for each $x_2\in E_1$, $(x_2,\overline{y_2})\in E$. 
Thus
\begin{align*}
\lambda^{-\delta}\lesssim&\; e^{-i\widetilde{\theta}(x_2,y_2)}\int f_1(x_1,x_2)e^{i\lambda q_1(x_1,x_2,y_2)-ix_1\theta(x_2,y_2)}
dx_1
\\
=&\int f_1(x_1,x_2)e^{i\varphi_1(x_1,x_2)}dx_1\;\;\;\textrm{ if }x_2\in E_1
\end{align*}
where 
\[\varphi_1(x_1,x_2)=\lambda q_1(x_1,x_2,\bar y_2)-x_1\theta(x_2,\bar y_2)-\tilde\theta(x_2,\bar y_2)\]
is a real-valued  polynomial in $x_1$ of degree $\le d$,
whose coefficients are measurable functions of $x_2$.

We would like to use this to conclude that $f_1$ has reasonably large inner product with $e^{-i\varphi_1}$. While this is not necessarily true, the following {\em extension argument} will prove sufficient for our purposes: for every $x_2\in\R$, choose $\theta^*(x_2)$ in a measurable way and such that 
$$e^{i\theta^*(x_2)}\int_\R f_1(x_1,x_2)e^{i\varphi_1(x_1,x_2)}dx_1\geq 0.$$
We can guarantee that $\theta^*\equiv 0$ on $E_1$.

Define $\phi_1(x_1,x_2):=\theta^*(x_2)+\varphi_1(x_1,x_2)$. Then 
$\phi_1$ is likewise a real-valued polynomial in $x_1$ of degree $\le d$,
whose coefficients are measurable functions of $x_2$. 
Now
$$\int_\R f_1(x_1,x_2)e^{i\phi_1(x_1,x_2)}dx_1\geq 0\;\;\;\textrm{ for every }x_2\in \R$$
while for any $x_2\in E_1$,
$$\int_\R f_1(x_1,x_2)e^{i\phi_1(x_1,x_2)}dx_1
\gtrsim \lambda^{-\delta}.$$
Therefore
since $|E_1|\gtrsim \lambda^{-\delta} $,
\begin{equation*}
|\langle f_1,e^{-i\phi_1}\rangle|\gtrsim\lambda^{-2\delta}.
\end{equation*}
Since $\|f_1\|_{L^2}=1$
and $f_1$ is supported in a fixed bounded set,
\begin{equation}\label{1}
\|f_1-\langle f_1, e^{-i\phi_1}\rangle e^{-i\phi_1}\|_2^2\leq (1-C\lambda^{-4\delta}).
\end{equation}

Let $A(\lambda)$ be the best constant in the inequality
$$|I(\lambda P; f_1,f_2,f_3)|\leq A(\lambda)\|f_1\|_{L^2}\|f_2\|_{L^2}\|f_3\|_{L^2}.$$ That $A(\lambda)$ is finite is an immediate consequence of the dual form of Young's convolution inequality and the fact that, in this context, $L^2\subset L^{3/2}$.
Now \eqref{1} implies
\begin{align*}
|I(\lambda P; f_1,f_2,f_3)|=&|I(\lambda P;f_1-\langle f_1, e^{-i\phi_1}\rangle e^{-i\phi_1},f_2,f_3)+I(\lambda P; \langle f_1, e^{-i\phi_1}\rangle e^{-i\phi_1},f_2,f_3)|\\
\leq& A(\lambda)\|f_1-\langle f_1, e^{-i\phi_1}\rangle e^{-i\phi_1}\|_2\|f_2\|_2\|f_3\|_2+|\langle f_1, e^{-i\phi_1}\rangle||I(\lambda P; e^{-i\phi_1},f_2,f_3)|\\
\leq& (1-C\lambda^{-4\delta})^{1/2}A(\lambda)+C|I(\lambda P; e^{-i\phi_1},f_2,f_3)|,
\end{align*}
Therefore
\begin{equation*}
A(\lambda)\le
(1-C\lambda^{-4\delta})^{1/2}A(\lambda)+C\sup_{\phi_1,f_2,f_3}|I(\lambda P; e^{-i\phi_1},f_2,f_3)|
\end{equation*}
where the supremum is taken over all functions $f_2,f_3$ supported in the specified regions
satisfying $\|{f_j}\|_{L^2}=1$,
and over all real-valued functions $\phi_1(x_1,x_2)$ which are polynomials
of degree $\le d$ with respect to $x_1$, with coefficients depending measurably on $x_2$.
Since $A(\lambda)<\infty$, it follows that
\begin{equation} \label{swallow}
A(\lambda)\le
C\lambda^{4\delta}\sup_{\phi_1,f_2,f_3}|I(\lambda P; e^{-i\phi_1},f_2,f_3)|.
\end{equation}
Therefore it suffices to prove that
\[|I(\lambda P; e^{-i\phi_1},f_2,f_3)|\le C\lambda^{-\eps}\|f_2\|_{L^2}\|f_3\|_{L^2}\]
for a certain $\eps>0$; for $\delta$ may then be chosen to equal $\eps/5$.

By repeating the above steps for
$g_2$ and $g_3$,
we conclude that it suffices to prove that
there exists $\eps>0$ such that
\begin{equation} \label{progress}
|I(\lambda P; e^{i\phi_1},e^{i\phi_2},e^{i\phi_3})|
\le C\lambda^{-\eps}
\end{equation}
uniformly for all $\lambda\ge 1$, all polynomials $P$
satisfying $\|P\|_{\text{nd}}=1$,
and all real-valued measurable functions $\phi_j(u_1,u_2)$
which are polynomials of degree $\le d$ with respect to $u_1$.

\section{Handling remainders}

In the last section we have reduced matters to the case where the $f_j$ are of the special form

\begin{displaymath}
\left\{ \begin{array}{ll}
f_1(x_1,x_2)=e^{i\phi_1(x_1,x_2)}\\
f_2(y_1,y_2)=e^{i\phi_2(y_1,y_2)}\\
f_3(x_1+y_1,x_2+y_2)=e^{i\phi_3(x_1+y_1,x_2+y_2)}.\\
\end{array} \right.
\end{displaymath}
where $\phi_j$ are partial polynomials in the sense described following \eqref{progress}.
Express
\begin{displaymath}
\left\{ \begin{array}{ll}
\phi_1(x_1,x_2)=\sum_{j=0}^d \theta_{1,j}(x_2)x_1^j\\
\phi_2(y_1,y_2)=\sum_{k=0}^d \theta_{2,k}(y_2)y_1^k\\
\phi_3(x_1+y_1,x_2+y_2)=\sum_{l=0}^d \theta_{3,l}(x_2+y_2)(x_1+y_1)^l\\
\end{array} \right.
\end{displaymath}
where
$\theta_{1,j}$, $\theta_{2,k}$,  $\theta_{3,l}$ are measurable and real-valued.
Also express $P(x,y)=\sum_{j,k} p_{jk}(x_2,y_2)x_1^jy_1^k$. Then
\begin{align*}
I(\lambda P;f_1,f_2,f_3)
=&\iint e^{i\lambda P}e^{i\phi_1}e^{i\phi_2}e^{i\phi_3}\eta dxdy\\
=&\iint\Big(\iint e^{i\sum_{j,k}\psi_{jk}(x_2,y_2)x_1^jy_1^k}\eta dx_1dy_1\Big)dx_2dy_2,
\end{align*}
where
\begin{displaymath}
\psi_{jk}(x_2,y_2)= \left\{ \begin{array}{ll}
\theta_{1,j} (x_2)& \textrm{if $k=0$}\\
0 & \textrm{if $k\neq0$}
\end{array} \right\}
+\left\{ \begin{array}{ll}
\theta_{2,k}(y_2) & \textrm{if $j=0$}\\
0 & \textrm{if $j\neq0$}
\end{array} \right\}
+{j+k\choose k}\theta_{3,j+k}(x_2+y_2)
+\lambda p_{jk}(x_2,y_2)
.\end{displaymath}

The desired bound $|I(\lambda P; e^{i\phi_1},e^{i\phi_2},e^{i\phi_3})|\le C\lambda^{-\eps}$
follows directly from Theorem~\ref{sp}, unless
there exists 
a measurable subset $E\subset \R^2$ of measure 
$|E|\gtrsim\lambda^{-\delta}$  such that
\begin{equation}\label{psi}
\sum_{(j,k)\ne (0,0)} |\psi_{jk}(x_2,y_2)|\lesssim\lambda^r,\;\;\; \forall (x_2,y_2)\in E.
\end{equation}
We may choose $\delta,r>0$ to be as small as may be desired for later purposes,
at the expense of taking $\eps$ sufficiently small in \eqref{progress}.

The proof of the following lemma will be given later.
\begin{lemma}\label{cousin}
Let $P:\R^2\rightarrow\R^D$ be a real vector-valued polynomial of degree $d$, and let $f,g:[0,1]\rightarrow\R^D$ be measurable functions. 
Let $E\subseteq [0,1]^2$ be a measurable subset of the unit square of Lebesgue measure $|E|=\epsilon>0$. Assume that 
\begin{equation}\label{3}
|f(x)+g(y)+P(x,y)|\leq 1\textrm{ for all }(x,y)\in E.
\end{equation}
Then there exist $\R^D$--valued polynomials 
$Q_1$ and $Q_2$ of degrees $\leq d$ and measurable sets $E_1,E_2\subseteq[0,1]$ such that 
\begin{displaymath}
\left\{ \begin{array}{ll}
|f(x)-Q_1(x)|\lesssim \epsilon^{-C}\textrm{ for }x\in E_1\\
|g(y)-Q_2(y)|\lesssim \epsilon^{-C}\textrm{ for }y\in E_2\\
|E_1|\geq c\epsilon, \;\; |E_2|\geq c\epsilon.\\
\end{array} \right.
\end{displaymath}
The constants $c,C\in\R^+$ depend only on $d$.
\end{lemma}

The phase estimates \eqref{psi}, together with Lemma~\ref{cousin}, allow us to control most of the terms $\theta_i$. Letting $k=0$, we have that
$$|\psi_{j0}(x_2,y_2)|=|\theta_{1,j}(x_2)+\theta_{3,j}(x_2+y_2)+\lambda p_{j0}(x_2,y_2)|\lesssim\lambda^r$$
if $1\leq j\leq d$ and $(x_2,y_2)\in E$. Since $|E|\gtrsim\lambda^{-\delta}$, Lemma \ref{cousin} implies that, for every $1\leq j\leq d$, there exists a real-valued polynomial $\widetilde{Q}_{1,j}$ of degree $\leq d$ such that
$$|\lambda^{-r}\theta_{1,j}(x_2)-\widetilde{Q}_{1,j}(x_2)|\lesssim (\lambda^{-\delta})^{-C}$$
whenever $x_2\in E_1$; here, $E_1\subset\R$ is a measurable subset which does not depend on $j$ and such that $|E_1|\gtrsim\lambda^{-\delta}$.
A similar conclusion can be drawn for each of the terms $\theta_{3,l}$ with $1\leq l\leq d$. Choosing $j=0$ we control the terms $\theta_{2,k}$ for $1\leq k\leq d$ in an analogous way. 

We conclude that, for every $1\leq j,k,l\leq d$,
\begin{displaymath}
\left\{\begin{array}{ll}
\theta_{1,j}(x_2)&=Q_{1,j}(x_2)+\widetilde{R_{1,j}}(x_2)\\
\theta_{2,k}(y_2)&=Q_{2,k}(y_2)+ \widetilde{R_{2,k}}(y_2)\\
\theta_{3,l}(x_2+y_2)&=Q_{3,l}(x_2+y_2)+ \widetilde {R_{3,l}}(x_2+y_2)\\
\end{array}\right.
\end{displaymath}
where $Q_{1,j}$, $Q_{2,k}$ and $Q_{3,l}$ are polynomials of degree $\leq d$, and the remainders satisfy 
\begin{displaymath}
\left\{\begin{array}{ll}
| \widetilde {R_{1,j}}(x_2)|\lesssim \lambda^{\beta}&\textrm{ if }x_2\in E_{1}, 
\\
| \widetilde {R_{2,k}}(y_2)|\lesssim \lambda^{\beta}&\textrm{ if }y_2\in E_{2}, 
\\
| \widetilde {R_{3,l}}(x_2+y_2)|\lesssim \lambda^{\beta}&\textrm{ if }x_2+y_2\in E_{3}, 
\end{array}\right.
\end{displaymath}
for certain measurable subsets $E_{1},E_{2},E_{3}\subset\R$
which satisfy
\[
|E_i|\gtrsim\lambda^{-\delta}.
\]
The parameter $\beta:={r+C\delta}$ is a function of $r,\delta$ and the constant $C=C(d)$ from Lemma \ref{cousin}.

We have estimates for the remainders $R_i$ in rather small sets only, 
but it is possible to reduce to the case in which these estimates hold globally,
via an extension argument similar to the one used in the previous section.
Set $Q_1(x)=\sum_{j=1}^d Q_{1,j}(x_2)x_1^j$ and $\widetilde {R_1}(x)=\sum_{j=1}^d \widetilde {R_{1,j}}(x_2)x_1^j$.
By modifying each $\widetilde {R_{1,j}}$ suitably at each point of the complement of $E_1$,
we produce a function $\Phi_1 = \theta_{1,0}+Q_1+ R_1$ 
such that $\theta_{1,0}$ is a measurable and real-valued function of $x_2$, $Q_1(x)$ is a polynomial function of $x\in\R^2$ of degree $\le d$,
$R_1(x_1,x_2)$ is a polynomial in $x_1$ of degree $\le d$ whose coefficients
are measurable functions of $x_2$,
$| R_1(x)|\lesssim\lambda^\beta$ for every $x\in\R^2$,
all functions are real-valued,
and
\begin{equation}
\langle e^{i\phi_1},e^{i\Phi_1}\rangle
\gtrsim \lambda^{-\delta}.
\end{equation}
By the same argument used to reduce from general $f_j$ to $e^{i\phi_j}$ in \eqref{swallow},
\begin{equation}
A(\lambda) \lesssim \lambda^{C\delta}
\sup_{\Phi_1,\phi_2,\phi_3}
|I(\lambda P; e^{i\Phi_1},e^{i\phi_2},e^{i\phi_3})|
\end{equation}
where the supremum is taken over all $\Phi_1,\phi_2,\phi_3$ of the above form.
This argument can be repeated twice more to give
\begin{equation}
A(\lambda) \lesssim \lambda^{C\delta}
\sup_{\Phi_1,\Phi_2,\Phi_3}
|I(\lambda P; e^{i\Phi_1},e^{i\Phi_2},e^{i\Phi_3})|
\end{equation}
where each of the functions $\Phi_i$ shares the properties indicated above for $\Phi_1$.

Now
\begin{multline*}
I(\lambda P;e^{i\Phi_1},e^{i\Phi_2},e^{i\Phi_3})
=\iint e^{i\theta_{1,0}(x_2)}e^{i\theta_{2,0}(y_2)}e^{i\theta_{3,0}(x_2+y_2)}
\\
\cdot
\Big(\iint e^{i\lambda P(x,y)} e^{iQ_1(x)}e^{iQ_2(y)}e^{iQ_3(x+y)}
e^{i(R_1(x)+R_2(y)+R_3(x+y))}\eta dx_1dy_1\Big)dx_2dy_2.
\end{multline*}

Let $\widetilde{P}:=P+\lambda^{-1}{Q_1}+\lambda^{-1}{Q_2}+\lambda^{-1}{Q_3}$. Since $[P]=[\widetilde{P}]$, 
$\|P\|_{nd}=\| \widetilde {P}\|_{nd}$. We are left with:
$$I(\lambda P)= \iint e^{i\theta_{1,0}(x_2)}e^{i\theta_{2,0}(y_2)}e^{i\theta_{3,0}(x_2+y_2)}\Big(\iint e^{i\lambda \widetilde{P}(x,y)}e^{i(R_1(x)+R_2(y)+R_3(x+y))}\eta dx_1dy_1\Big)dx_2dy_2$$
where all functions in the exponents are real-valued, 
$\widetilde{P}$ is a polynomial of degree $\leq d$ such that $\| \widetilde {P}\|_{nd}=\|P\|_{nd}=1$, the $\theta_{j,0}$ and the $R_j$ are measurable functions, and the remainders $R_j(u_1,u_2)$ are polynomial functions 
of $u_1$ of degrees $\le d$, which satisfy $|R_j(u)|\lesssim \lambda^\beta$ for all $u\in\R^2$.

\section{The end of the proof}

Decompose
\begin{equation} \label{decomposeP}
\widetilde{P}=P_0+P^*
\qquad\text{where}\qquad
P_0(x_2,y_2)=\widetilde{P}(0,0,x_2,y_2) 
\end{equation}
and $P^*=\widetilde{P}-P_0$. 

Since
$$\psi_{00}(x_2,y_2)=\lambda P_0(x_2,y_2)+\theta_{1,0}(x_2)+\theta_{2,0}(y_2)+\theta_{3,0}(x_2+y_2),$$ 
we can write

\begin{equation}\label{expr}
I(\lambda P)= \iint e^{i\psi_{00}(x_2,y_2)}\Big(\iint e^{i\lambda P^*(x,y)}e^{i(R_1(x)+R_2(y)+R_3(x+y))}\eta dx_1dy_1\Big)dx_2dy_2.
\end{equation}

Our main assumption, namely that $P$ is nondegenerate with respect to the projections $\{\pi_j\}_{j=1}^3$, has not 
yet come into play. To apply it, we need a lemma:
\begin{lemma}
For any $d\in{\mathbb N}$ there exists $c>0$ with the following property.
Let 
$\tilde P:\R^4\rightarrow\R$ be any real-valued polynomial of degree $\leq d$. 
Decompose $\tilde P=P_0+P^*$ as in \eqref{decomposeP}.
Then
$$\Big\|\|P_{(x_2,y_2)}^*\|_{nc}\Big\|_{\P(d)}+\|P_0\|_{nd}\geq c\|\tilde P\|_{nd}.$$
\end{lemma}

The expression $\|P_0\|_{nd}$ in this lemma has two natural interpretations, but these define
the same quantity since
\begin{multline} \label{twopossible}
\min_{\textrm{deg} (p_i)\leq d}\|P_0(x_2,y_2)+p_1(x_1,x_2)+p_2(y_1,y_2)+p_3(x_1+y_1,x_2+y_2)\|_{\P(d)}\\
=\min_{\textrm{deg} (q_i)\leq d}\|P_0(x_2,y_2)+q_1(x_2)+q_2(y_2)+q_3(x_2+y_2)\|_{\P(d)}.
\end{multline}
The two inequalities implicit in this equality are obtained by setting $q_j(t)=p_j(0,t)$,
and by setting $p_j(t_1,t_2)=q_j(t_2)$, respectively.

\begin{proof}
The left-hand side defines a seminorm on the finite-dimensional vector space of
polynomials of degrees $\le d$ modulo degenerate polynomials,
so it suffices to show that if 
$\Big\|\|P_{(x_2,y_2)}^*\|_{nc}\Big\|_{\P(d)}$ vanishes then $P^*$ is degenerate,
and correspondingly for $P_0$.  For $P_0$ this is a tautology, in view of \eqref{twopossible}.
On the other hand,
$P^*(x,y)=\sum_{(j,k)\ne(0,0)} q_{j,k}(x_2,y_2)x_1^jy_1^k$
where $q_{j,k}$ are uniquely determined polynomials,
and
$\Big\|\|P_{(x_2,y_2)}^*\|_{nc}\Big\|_{\P(d)}=0$
if and only if $q_{j,k}\equiv 0$ for each $j,k$.
Thus $P^*\equiv 0$, so in particular, $P^*$ is degenerate.
\end{proof}

Therefore there exists a constant $c_d>0$ such that 
$\Big\|\|P_{(x_2,y_2)}^*\|_{nc}\Big\|_{\P(d)}\geq c_d$
or
$\|P_0\|_{nd}\geq c_d$.

\subsection{Case 1: $\Big\|\|P_{(x_2,y_2)}^*\|_{nc}\Big\|_{\P(d)}\geq c_d$.  }
We have that
\begin{multline*}
I(\lambda P)=\iint e^{i(\lambda P_0(x_2,y_2)+\theta_{1,0}(x_2)+\theta_{2,0}(y_2)+\theta_{3,0}(x_2+y_2))}
\\
\cdot\Big(\iint e^{i\lambda P^*(x,y)}e^{i(R_1(x)+R_2(y)+R_3(x+y))}\eta dx_1dy_1\Big)dx_2dy_2,
\end{multline*}
where by Theorem \ref{sp}, the absolute value of the inner integral is 
$$\lesssim (\lambda \|P^*_{(x_2,y_2)}+\lambda^{-1}R_{1,(x_2)}+\lambda^{-1}R_{2,(y_2)}+\lambda^{-1}R_{3,(x_2+y_2)}\|_{nc})^{-1/d}.$$
Since $|R_j|\le\lambda^\beta$ and $\beta<1$,
we have that, if $\lambda$ is large enough, then
$$\|P^*_{(x_2,y_2)}+\lambda^{-1}R_{1,(x_2)}+\lambda^{-1}R_{2,(y_2)}+\lambda^{-1}R_{3,(x_2+y_2)}\|_{nc}\gtrsim \|P^*_{(x_2,y_2)}\|_{nc}$$
for every $(x_2,y_2)\in \R^2$. We conclude that the absolute value of the inner integral is 
$$\lesssim (\lambda\|P^*_{(x_2,y_2)}\|_{nc})^{-1/d}.$$
If $(x_2,y_2)\in \R^2$ is such that 
$$\|P^*_{(x_2,y_2)}\|_{nc}\gtrsim\lambda^{-1+\tau}\textrm{ for some }\tau>0,$$ 
we get the desired decay. Otherwise, observe that (i) implies a sublevel set estimate of the form
$$|\{(x_2,y_2)\in\R^2:\|P^*_{(x_2,y_2)}\|_{nc}<\lambda^{-1+\tau}\}|\lesssim (\lambda^{-1+\tau})^{\delta}.$$ 
Therefore the contribution of the set of such points $(x_2,y_2)$ to the integral is small, and this concludes the analysis of Case 1.

\subsection{ Case 2: $\|P_0\|_{nd}\geq c_d$.} \label{case2}
By Fubini,
\begin{multline*}
I(\lambda P)=\iint \Big(\iint e^{i\lambda [P_0(x_2,y_2)+P^*(x,y)]} \\
\cdot e^{i(\theta_{1,0}+R_{1,(x_1)})(x_2)}e^{i(\theta_{2,0}+R_{2,(y_1)})(y_2)}e^{i(\theta_{3,0}+R_{3,(x_1+y_1)})(x_2+y_2)}\eta dx_2dy_2\Big)dx_1dy_1.
\end{multline*}
By Theorem \ref{l2}, the absolute value of the inner integral in the last expression is $\lesssim (1+\lambda^2\|P_0+P_{(x_1,y_1)}^*\|^2_{nd})^{-\rho}$, for some $\rho=\rho(d)>0$. It follows that 
$$|I(\lambda P)|\lesssim \iint (1+ \lambda^2\|P_0+P_{(x_1,y_1)}^*\|^2_{nd})^{-\rho}dx_1dy_1.$$
To handle this integral, note that $P^*(0,0,x_2,y_2)=0$ and hypothesis (ii) together imply that
$$|\{(x_1,y_1)\in\R^2: \|P_0+P^*_{(x_1,y_1)}\|^2_{nd}<\epsilon\}|\lesssim \epsilon^{\delta}$$ because $(x_1,y_1)\mapsto \|P_0+P_{(x_1,y_1)}^*\|_{nd}^2$ is a polynomial of degree $\leq 2d$.
An argument entirely analogous to the one used to prove Lemma~\ref{first} concludes the analysis.

\section{Higher Dimensions and Generalization} \label{section:higher}




So far, we have only discussed the case where the domain of $P$ is $\R^4$, but
$\R^{2\kappa}$ for $\kappa>2$ is treated in essentially the same way. Now write
$x=(x',x_\kappa)$, $y=(y',y_\kappa)$
where $x',y'\in\R^{\kappa-1}$.
The proof proceeds by induction on $\kappa$.
The only significant change in the proof is that
in Case 2 of the final step of the proof, since $\R^{\kappa-1}$ is no longer $\R^1$,
Theorem~\ref{l2} does not apply; instead, one simply invokes the induction hypothesis.

Our result may be generalized to include arbitrary smooth phases and not just polynomial ones. The details are a straightforward modification of those in \cite{G} and will therefore not be included.

\section{Proof of Lemma~\ref{cousin}}

It remains to prove Lemma~\ref{cousin}. The proof  will rely on the following related, but simpler, result. 
Let $X$ be a normed linear space.
We write $|x|$ to denote the norm of a vector in $X$.
\begin{lemma}\label{frust}
Let $\Omega,\Omega'$ be probability spaces with measures $\mu,\mu'$, and let $f,f'$ be 
$X$-valued functions defined on these spaces. Let $0<r<1$ and $R\in (0,\infty)$. Let $E\subset\Omega\times\Omega'$ satisfy $(\mu\times\mu')(E)\geq r$. Suppose that 
$$|f(x)-f'(x')|\leq R\textrm{ for all }(x,x')\in E.$$
Then there exist $a\in X$ and $G\subset\Omega$,$G'\subset\Omega'$ such that
$$\mu(G)\geq cr$$
$$\mu'(G')\geq cr$$
$$|f(x)-a|\leq CR\textrm{ for all }x\in G$$
$$|f'(x')-a|\leq CR\textrm{ for all }x'\in G'.$$
Here $c,C$ are certain absolute constants. 
\end{lemma}

\begin{proof}
$|E|$ will denote $(\mu\times\mu')(E)$.
Let $\pi_1:\Omega\times\Omega'\rightarrow\Omega$ and $\pi_2:\Omega\times\Omega'\rightarrow\Omega'$ denote the canonical projections. For $x\in\pi_1(E)$ and $x'\in\pi_2(E)$, consider the slices
\begin{displaymath}
\left\{ \begin{array}{ll}
E^x:=\{x'\in\Omega': (x,x')\in E\}\\
E^{x'}:=\{x\in\Omega: (x,x')\in E\}.\\
\end{array} \right.
\end{displaymath}
By Fubini, $E^x$ is $\mu'$-measurable for $\mu$-a.e. $x$ and $E^{x'}$ is $\mu$-measurable for $\mu'$-a.e. $x'$.

{\em Claim.} There exists $(x_0,x_0')\in E$ such that
\begin{displaymath}
\left\{ \begin{array}{ll}
G:=E^{x_0'}\subset\Omega\textrm{ is }\mu\textrm{-measurable and }\mu(G)\geq cr\\
G':=E^{x_0}\subset\Omega'\textrm{ is }\mu'\textrm{-measurable and }\mu(G')\geq cr.\\
\end{array} \right.
\end{displaymath}
Assuming the claim, we have that:
\begin{displaymath}
\left\{ \begin{array}{ll}
|f(x)-f'(x_0')|\leq R & \textrm{for every $x\in G$,}\\
|f(x_0)-f'(x_0')|\leq R\\
|f(x_0)-f'(x')|\leq R & \textrm{for every $x'\in G'$.}\\
\end{array} \right.
\end{displaymath}
Let  $a:=\frac{f(x_0)+f'(x_0')}{2}$. Then, for any $x\in G$,
$$|f(x)-a|\leq |f(x)-f'(x_0')|+|f'(x_0')-a|\leq \frac{3}{2}R,$$ and similarly for $x'\in G'$.

To prove the claim, start by assuming that $E^x$ and $E^{x'}$ are measurable for every $(x,x')\in E$.
Express $E$ as a disjoint union $E=\G\cup\B$, where
\begin{displaymath}
\left\{ \begin{array}{ll}
\G=\{(x,x')\in E: \mu'(E^x)\geq \frac{r}{4}\textrm{ and } \mu(E^{x'})\geq \frac{r}{4}\}\\
\B=\B_1\cup\B_2:=\{(x,x')\in E: \mu'(E^x)<\frac{r}{4}\}\cup\{(x,x')\in E: \mu(E^{x'})<\frac{r}{4}\}.\\
\end{array} \right.
\end{displaymath}

\noindent We prove the stronger statement $|\G|>0$. Suppose on the contrary that $|E|=|\B|$. Then 
$$r\leq |E|=|\B|=|\B_1\cup\B_2|\leq|\B_1|+|\B_2|, $$
whence $|\B_1|\geq\frac{r}{2}$ or $|\B_2|\geq\frac{r}{2}$. Without loss of generality assume that the former holds. Then 
$$\Big|\Big\{(x,x')\in E: \mu'(E^x)\geq \frac{r}{4}\Big\}\Big|\leq |E|-\frac{r}{2}.$$
But then, defining $S_1:=\{x\in\pi_1(E): \mu'(E)\geq\frac{r}{4}\}$ and $S_2:=\pi_1(E)\setminus S_1$, we have that
\begin{align*}
|E|=\int_{\pi_1(E)}\mu'(E^x)d\mu(x)=&\int_{S_1}\mu'(E^x)d\mu(x)+\int_{S_2}\mu'(E^x)d\mu(x)\\
\leq&\Big(|E|-\frac{r}{2}\Big)+\frac{r}{4},
\end{align*}
a contradiction.
\end{proof}

\begin{proof} [Proof of Lemma \ref{cousin}]
Let $A$ be the norm of $P$ in the quotient space of $\R^D$--valued 
polynomials of degree $\leq d$ modulo degenerate polynomials with respect to the pair of 
projections $(x,y)\mapsto x$ and $(x,y)\mapsto y$. 
It is well known \cite{S} that a decay bound of the form \eqref{decay}
holds for these projections, that is,
for any $d$ there exists $\rho>0$ such that
for any compact sets $K,K'\subset\R$
there exists $C<\infty$ such that
$|\int_{\R^2}e^{iQ(x,y)}f(x)g(y)\,dx,dy|\le C(1+\|Q\|_{\text{nd}})^{-\rho}
\|f\|_2\|g\|_2$ 
for all functions $f,g$ supported on $K,K'$ respectively, 
for all real-valued polynomials $Q$ of degree $\le d$.
This in turn, applied to the individual components of $P$, implies a sublevel set inequality
of the form
$$|E|\leq C_\gamma A^{-\gamma},$$
where $C_\gamma=\frac{C}{1-\gamma}$, $C$ is an absolute constant, and $\gamma$ depends only on $d$; 
see for instance the discussion in \cite{C} for the simple derivation.

So 
$A\le
C_\gamma\epsilon^{-1/\gamma}$, that is, 
$$\inf_{\textrm{deg }p,q\leq d}\sup_{(x,y)\in [0,1]^2} |P(x,y)-p(x)-q(y)|\leq
C_\gamma \epsilon^{-1/\gamma}.
$$

\noindent The inf is actually a minimum, so there exist polynomials $p$ and $q$ of degree $\leq d$ such that 
$$\sup_{(x,y)\in [0,1]^2} |P(x,y)-p(x)-q(y)|\leq
C_\gamma \epsilon^{-1/\gamma}.$$
In particular, for $(x,y)\in E$, we have that
\begin{align*}
|(f(x)+p(x))+(g(y)+q(y))|\leq& |f(x)+g(y)+P(x,y)|+|p(x)+q(y)-P(x,y)|\\
\leq& 1+
C_\gamma \epsilon^{-1/\gamma}
=
C'_\gamma\epsilon^{-1/\gamma}.
\end{align*}
Apply lemma \ref{frust} to conclude the existence of $a\in\C$ and $E_1\subset [0,1]$ such that $|E_1|\geq \frac{\epsilon}{4}$ and 
$$|f(x)+p(x)-a|\leq C_\gamma\epsilon^{-1/\gamma}=C_\gamma |E|^{-1/\gamma}\;\;\;\;\;\textrm{for every }x\in E_1.$$
Proceeding similarly for $g+q$ completes the proof.
\end{proof}

\end{document}